\documentclass{amsart}
\usepackage{bbm}
\usepackage{amsmath,amsthm,amsfonts,amssymb,amscd,mathrsfs}
\usepackage{psfrag,graphicx}

\usepackage{epic,eepic}
\usepackage{color}
\usepackage{ebezier}


\theoremstyle{plain}
\newtheorem{main}{Theorem}

\newtheorem{maincor}[main]{Corollary}
\newtheorem{Thm}{Theorem}[section]
\newtheorem{Lem}[Thm]{Lemma}
\newtheorem{Prop}[Thm]{Proposition}

\theoremstyle{remark}

\newtheorem{Rem}[Thm] {Remark}

\newtheorem{Que}[Thm] {Question}

\long\def\begcom#1\endcom{}

\newcommand{\length}{\operatorname{\length}}

\def\length{\operatorname{length}}

\def\Int{\operatorname{Int}}

\def\vep{\varepsilon}

\begin{document}

\title[Uniform tail entropy  for real analytic  maps ]
      { Uniform tail entropy  for real analytic  maps }

\author[Gang Liao]
{ Gang Liao $^*$}

\email{liaogang@impa.br}

\thanks{$^{*}$ IMPA, Estrada D. Castorina 110, Jardim Bot\^anico,
22460-320 Rio de Janeiro, Brazil.}

\thanks{$^{*}$ This research is supported by CNPq-Brazil and  CPSF(\#2014M560007)-China.}


\maketitle


\begin{abstract}
Let $M$ be a compact real analytic manifold of finite dimension. There is
a function $a: (0,+\infty)\to [0,+\infty)$ with $\lim_{t\to0}a(t)=0$
such that, the tail entropy $h^{*}(f,\vep)$ of any real analytic map $f$
on $M$ is uniformly bounded above by the scale $a(\vep)$.
\end{abstract}


\section{Introduction}

In topological dynamics, a system $(f,M)$ is understood to be a
continuous map $f$ acting on a compact metric space $M$.  We first
recall the definition of topological entropy, see \cite{Walters}.
For any compact subset $\Lambda\subset M$, an observable scale
$\vep>0$ and time $n\in \mathbb{N}$, a subset $K\subset
 \Lambda$ is said to $(n,\vep)$-span $\Lambda$ if for any $x\in
 \Lambda$ there exists $y\in K$ which stays close to $x$ within the same scale $\vep$ for all times $i\in [0,n)$, i.e., $d(f^ix,f^iy)\leq
 \vep$, $\forall\,i\in [0,n)$. Let
 $r_n(f,\Lambda,\varepsilon)$ denote the smallest cardinality of any
$(n,\vep)$-spanning set of $\Lambda$. The $\vep$-topological entropy
of $\Lambda$ is defined to be the exponential growth rate of
$(n,\vep)$-orbits:
$$h(f,\Lambda,\varepsilon)=\limsup_{n\rightarrow\infty}\frac{1}{n}\log
r_n(f,\Lambda,\varepsilon).$$ Letting $\vep\to 0$, define the
topological entropy of $f$ on $\Lambda$ by
$$h(f,\Lambda)=\lim_{\varepsilon\rightarrow0}h(f,\Lambda,\varepsilon).$$
For simplicity, $h(f,\varepsilon)=h(f,M,\varepsilon)$,
$h(f)=h(f,M)$.

 Given $x\in M$, $n\in \mathbb{N}$,
 the $n$-step dynamical ball $B_{n}(f, x,\varepsilon)$
is the set of all points $y\in M$ such that
$$d(f^iy,\,f^ix)<\varepsilon,~~i=0, 1,\cdots,n-1.$$ Let
$B_{\infty}(f,x,\vep)=\cap_{n\in \mathbb{N}}B_{n}(f,
x,\varepsilon)$. Define the $\vep$-tail entropy
$$h^{*}(f,\varepsilon)=\sup_{x\in M}h(f,\, B_{\infty}(f,x,\varepsilon)).$$
We say $f$ is entropy expansive \cite{Bow72b}   if there exists
$\delta>0$ such that
$$h^{*}(f,\delta)=0.$$

Recall the following two related notions:

 (1) as a special case, $f$ is called positive expansive if there exists $\delta>0$ such that
$$B_{\infty}(f,x,\delta)=\{x\},\quad \forall\,x\in M;$$

(2) more generally, $f$ is called asymptotically  entropy expansive
\cite{Mis76} if
$$\lim_{\delta\to 0}h^{*}(f,\delta)=0.$$
In \cite{Bow72b}  the entropy expansiveness is proposed by Bowen as
a topological condition guaranteeing the upper-semicontinuity of
measure theoretic entropy, which together with the variational
principle implies the existence of maximal measures, i.e., invariant
measures with measure theoretic entropy equal to the topological
entropy.  Moreover, it is asked by Bowen \cite{Bow72b}  whether
there exist  diffeomorphisms which are not entropy expansive.  For any
$1<r<\infty$, Misiurewicz \cite{Mis73} constructed $C^r$
diffeomorphisms without any maximal measure hence not entropy
expansive and also not asymptotically entropy expansive. In the hope
of finding a characterization of the existence of maximal measures,
the discussion later turned to some weaker form of expansivity
(called local entropy expansivity), see Chapter 20 of \cite{Denker}.

To study the uniform properties of tail entropy, Boyle and
Downarowicz developed the theory of entropy structure in
\cite{Boyle-Downarowicz}. Asymptotic entropy expansiveness is
exactly equivalent to  the property that  entropy structure converges uniformly to
measure theoretic entropy for all invariant measures (see Theorem
8.6 of \cite{Boyle-Downarowicz}). Combining  insights from the work of
 Yomdin \cite{Yom87, Yom87-2}, Gromov \cite{Gr85}, Newhouse \cite{New89}
 one realizes that the entropy structure actually depends on the smoothness of
dynamical systems. The larger  the order of differentiability is,
the smaller the tail entropy is. In contrast to finite
differentiability, Buzzi \cite{Buz97} established that all
$C^{\infty}$ maps
  are asymptotically entropy expansive. Assuming real analyticity, we can
expect a more uniform estimate for tail  entropy.
By polynomial approximation and Bernstein inequality,  Yomdin
\cite{Yom91} has obtained that for any surface real
 analytic map $f$,  $h^{*}(f,\vep)$ is  bounded above  by
$C(f)\log|\log\vep|/|\log\vep|$ with the constant $C(f)$ depending on $f$ but independent of $\vep$.   A
further question is the following

\begin{Que}\label{entropy expansive}
Is every  real analytic map entropy expansive?
\end{Que}

\begin{Rem}It is known that uniformly hyperbolic systems are
expansive\footnote{When $f$ is a homeomorphism, $f$ is called to be expansive if
$B_{\pm\infty}(f,x,\vep):=\{y\in M\mid d(f^iy,f^ix)<\vep,\quad
\forall\,i\in \mathbb{Z}\}=\{x\}$, see \cite{Utz}. } (hence entropy expansive by Corollary 2.3 of \cite{Bow72b}). On the contrary, there exist polynomial maps on $\mathbb{R}^2\cup\{\infty\}$ with non-uniformly hyperbolic
behavior which are non-expansive and analytic except at $\infty$,
see Milnor \cite{Milnor}.

\end{Rem}
To generalize the result in \cite{Yom91} to higher dimensions,  Yomdin in \cite{Yom08} attempted to improve the analytic reparametrization of semi-algebraic sets, which has been realized for two dimensional case in \cite{Yom08}.  However,
higher dimensional analytic reparametrization involves complicated
analysis on the singularities of semi-algebraic sets  and   so far no
quantitative expressions as in \cite{Yom08} have been given.

In the present note  we attempt to give some uniform estimates for
the tail entropy of real analytic maps providing a partial answer to the Question \ref{entropy
expansive}. It mainly supplements the original approach in
\cite{Yom91} with two ideas: (i) the Cauchy formula to bound higher
derivatives; (ii) choice of the iterate to which one applies
Yomdin-Gromov's algebraic lemma \cite{Yom87-2,Gr85} depending on the
scale $\vep$.

\begin{main}\label{Thm1}Let $M$ be a compact real analytic manifold of finite dimension.  There is a
function $a: (0,+\infty)\to [0,+\infty)$ with $\lim_{t\to 0} a(t)=0$ depending only on the dimension of $M$, such that for any real analytic map $f$ of $M$ there exists $C(f)>0$
such that for any $\vep>0$,
$$h^{*}(f,\vep)\leq C(f)a(\vep).$$
\end{main}
\begin{Rem}The constant $C(f)$ in Theorem \ref{Thm1} can be chosen locally uniformly with respect to the complexification topology on the set of real analytic maps(its definition will be given in the proof). \end{Rem}
Noting that $h(f)-h(f,\vep)\leq h^{*}(f,\vep)$ , by  Theorem 2.4
of \cite{Bow72b},  we have the following corollary that gives the
approximation rate of $\vep-$entropy $h(f,\vep)$ in the calculation
of topological  entropy.
\begin{maincor}
Let $M$ be a compact real  analytic manifold of finite dimension. There is
a function $a: (0,+\infty)\to [0,+\infty)$ with $\lim_{t\to 0}
a(t)=0$  depending only on the dimension of $M$ such that, for any real analytic map $f$ of $M$ there exists
$C(f)>0$ such that for any $\vep>0$,
$$h(f)-h(f,\vep)\leq C(f) a(\vep).$$
\end{maincor}

\noindent{\it Acknowledgement.}  I am grateful to Martin Andersson, David Burguet, Marcelo Viana,
Jiagang Yang and Yosef Yomdin  for their helpful comments in the
preparation of this manuscript. And thanks a lot to the referees for
their helpful suggestions which have greatly improved the exposition and the original proofs.

\section{Proof of  theorem \ref{Thm1}}

Let $m=\dim M$.  For any $r>0$ we denote a standard cube
$$Q_r=\{(x_1,\cdots,x_m)\in \mathbb{R}^m\mid |x_i|\leq r,\,\,1\leq
i\leq m\}.$$ We say that a map $P$ is $C^s$ $(s\in \mathbb{N})$ on
the unit closed cube $Q_1$ if it is $C^s$ in the interior of $Q_1$
and all  derivatives  of order up to $s$ can be continuously
extended to the boundary of $Q_1$. The $C^s$ size of $P$ is
\begin{eqnarray*}
\|P\|_s=\sup\,\{\|d^kP(x)\|,\quad 1\leq k\leq s,\,\,x\in \Int Q_1\},
\end{eqnarray*}
where $\|d^kP\|$ denotes the supremum norm of the derivative $d^kP$.

 For an analytic map\footnote{Throughout this note, analytic map means real analytic map.} $f$,  there
is a  finite system of local charts $\{(U_1,\gamma_1),\cdots,
(U_k,\gamma_k)\}$ with open sets $U_i$ and invertible maps
$\gamma_i: B(0,1)\rightarrow U_i$ such that

\begin{itemize} \item if $f(U_i)\cap U_j\neq \emptyset$, then $f_{i,j}:=\gamma_j^{-1}\circ
f\circ\gamma_i\mid_{\gamma_i^{-1}( U_i\cap f^{-1}(U_j))}$ is the sum of its Taylor
series. \end{itemize} Such charts are called analytic charts. Noting
that $\{U_i\cap f^{-1}(U_j)\}$ constitutes an open cover of $M$, by the compactness of  $M$, we can choose $\rho>0$ to be the supremum of $r\leq 1/2$ such that for any $x\in M$ there exist $U_i, U_j$ with $B(\gamma^{-1}_i(x),r)\subset \gamma^{-1}_i(U_i\cap f^{-1}(U_j))$ and $f_{i,j}\mid{B(\gamma^{-1}_i(x),r)}$ can be extended to a complex Taylor series in $\widetilde{B}(\gamma_i^{-1}(x),r)$, where $\widetilde{B}(\gamma_i^{-1}(x),r)=\{y\in
\mathbb{C}^m\mid |y-\gamma_i^{-1}(x)|<r\}$. We say
$\rho=\rho(f)$ is the complex analytic radius of $f$.  Define
$$M_0=\sup_{i,j,x}|f_{i,j}\mid_{
\widetilde{B}(\gamma_i^{-1}(x),\rho)}|<\infty,$$
$$L(f)=\sup_{i,j,x}\|Df_{i,j}\mid_{
\widetilde{B}(\gamma_i^{-1}(x),\rho)}\|<\infty,$$ where the norm $|y|:=\max_{1\leq i\leq m}|y_i|$
whenever $y=(y_1,\cdots,y_m)\in \mathbb{C}^m$.
For $\vep,r>0$, we define the $(\vep,r)$-complexification neighborhood $\mathcal{N}(\vep,r;f;\{(U_1,\gamma_1),\cdots,
(U_k,\gamma_k)\})$ of $f$ by the following set
$$\Big{\{}g\mid \rho(g)\geq r,\,\,\sup_{i,j,x}\big{\{}|(g_{i,j}-f_{i,j})\mid_{
\widetilde{B}(\gamma_i^{-1}(x),r)}|,\,\,\|D(g_{i,j}-f_{i,j})\mid_{
\widetilde{B}(\gamma_i^{-1}(x),r)}\|\big{\}}<\vep\Big{\}}.$$
The complexification topology on the space of real analytic maps is generated by these $(\vep,r)$-complexification neighborhoods.

For simplicity, given any $x$ belonging to some $U_i\cap f^{-1}(U_j)$,
we identify $x$ with $\gamma^{-1}_i(x)$. By this identification in analytic charts,
we can suppose $B(x,\rho)\subset \mathbb{R}^m$, $\widetilde{B}(x,\rho)\subset \mathbb{C}^m$
and further, we can identify $f$ with $f_{i,j}$ in $\widetilde{B}(x,\rho)$.

Fix a positive real number $L_0>1$.  In order to construct the
universal function $a$, we first consider the case for an analytic
map $f: M\rightarrow M$ with $L(f)<L_0$ (see the statement in Lemma
\ref{small expanding}).

 Fix a positive  integer $n$. Since
$L(f)< L_0$, we have that
$$f^i(\widetilde{B}(x,L_0^{-n}\rho))\subset \widetilde{B}(f^i(x),\rho),\quad 0\leq i\leq n,$$
in the scale of local complex analytic radius.
 For $z=(z_1,\cdots, z_m)\in \mathbb{C}^m$,
write $D(z,r)=C(z_1,r)\times\cdots \times C(z_m,r)$, where
$C(z_i,r)\subset \mathbb{C}$ is the circle centered at $z_i$ with
radius $r$ in the complex plane $\mathbb{C}$.   By Cauchy's formula,
for $y\in \,\Int \,(D(x,\frac{1}{\sqrt{m}}\rho L_0^{-n}))\subset
\widetilde{B}(x,\rho L_0^{-n})$ and for any  multi-index
$\alpha=(\alpha_1,\cdots,\alpha_m)$ $(\alpha_i\in \mathbb{Z}\,\,
\text{and}\,\, \alpha_i\geq0 )$,
$$\partial^{\alpha}f^n(y)=\frac{\alpha!}{(2\pi i)^{m}}
\int_{ D(x,\frac{1}{\sqrt{m}}\rho L_0^{-n})}\frac{f^n(z)}{(z_1-y_1)^{\alpha_1+1}\cdots
(z_m-y_m)^{\alpha_m+1}}dz.$$   In particular, when $ y\in
\widetilde{B}(x,\frac{1}{2\sqrt{m}}\rho L_0^{-n})$,\,
$$d(y_i,C(x_i,\frac{1}{\sqrt{m}}\rho L_0^{-n}))\geq
\frac{1}{2\sqrt{m}}\rho L_0^{-n},\quad 1\leq i\leq m.$$ Therefore,
\begin{eqnarray*}
|\partial^{\alpha}f^n(y)|&\leq&
\frac{\alpha!}{(2\pi)^m}(\frac{1}{2\sqrt{m}}\rho
L_0^{-n})^{-|\alpha|-m} ( 2\pi \frac{1}{\sqrt{m}}\rho L_0^{-n})^m
M_0\\[2mm]
&=&  \alpha !\cdot 2^m\cdot
M_0\cdot(\frac{2\sqrt{m}L_0^n}{\rho})^{|\alpha|}.
\end{eqnarray*}
 Applying Stirling's
approximation, for $|\alpha|\geq 1$  we have
\begin{eqnarray}
\label{estimate of derivative} |\partial^{\alpha}f^n(y)|&\leq&
|\alpha|^{|\alpha|+\frac12}e^{1-|\alpha|}\cdot 2^m\cdot
M_0\cdot(\frac{2\sqrt{m}L_0^n}{\rho})^{|\alpha|}\\[2mm]&\leq&
|\alpha|^{|\alpha|+\frac12}\cdot 2^m\cdot
M_0\cdot(\frac{2\sqrt{m}L_0^n}{\rho})^{|\alpha|}\nonumber.
\end{eqnarray}
  Take
  $b_n=L_0^nn^2$ and write $g=f^n$, $\delta_n=b^{-1}_n$.
  There exists $N\geq 3$ such that for any $ n\geq N$,
  \begin{eqnarray}\label{large n}\delta_n\leq \min\{(4\sqrt{m}L_0^n\rho^{-1}n)^{-1},
   \,\,\rho L_0^{-n} \}\quad \text{and}\quad \frac{\log(4\sqrt{m+1}\rho^{-1}n^2)}{n}<1.\end{eqnarray}
   Along the orbit
   $\{g^i(x)\mid i\in \mathbb{Z}\}$, we define $g_{n,i}: B(0,2)\rightarrow \mathbb{R}^m$ by
$$g_{n,i}(t)=b_n(g(b_n^{-1} t+g^{(i-1)}(x))-g^{i}(x)).$$
\begin{Lem}\label{lem estimate of derivative}
There is a constant $C_0=C_0(m,L_0,M_0,\rho)$, independent of $n$, such
that
\begin{eqnarray*}\max_{t\in B(0,2)}\,\|d^kg_{n,i}(t)\|&\leq& C_0L_0^n, \quad 1\leq k\leq n,\quad\forall\, i\in \mathbb{N},\quad \forall\,n\geq 1.\end{eqnarray*}
 \end{Lem}\begin{proof}Note that \begin{eqnarray*}\max_{t\in B(0,2)}\,\|d^kg_{n,i}(t)\|&\leq&
b_n^{-k+1}\max_{y\in B(f^{n(i-1)}(x),2b_n^{-1})}\|d^kf^n(y)\|.
\end{eqnarray*} We only need to calculate the bound for $n\geq N$. By formulas (\ref{estimate of derivative}) and (\ref{large n}),  for
$1\leq k\leq n$ and $n\geq N$, we have
\begin{eqnarray*}\max_{t\in B(0,2)}\,\|d^kg_{n,i}(t)\|&\leq& [(\frac{2\sqrt{m}}{\rho}L_0^n)^{-k+1}
(2n)^{-k+1}]\cdot
 [2^m\cdot
M_0\cdot k^{k+\frac12}(\frac{2\sqrt{m}}{\rho}L_0^n)^{k}]\\[2mm]
&\leq & \frac{2^{m+1}\sqrt{m}M_0}{\rho}L_0^n\cdot
(2n)^{-k+1}k^{k+\frac12}.
\end{eqnarray*}
Considering the function $q(k)=(2n)^{-k+1}k^{k+\frac12}$,  for $n\geq N\geq 3$, we can compute the derivatives $$ \frac{d^2}{dk^2}\log
q(k)=\frac1k-\frac{1}{2k^2}\geq 0, \quad 1\leq k\leq n.$$ Hence,
\begin{eqnarray*}\max_{1\leq k\leq n}\,q(k) =\max\{q(1),q(n)\}=
\max\{1,\frac{n^{\frac32}}{2^{n-1}}\},\end{eqnarray*}which implies
\begin{eqnarray*}\sup_{n\geq N}\,\max_{1\leq k\leq
n}\,q(k)<\infty.\end{eqnarray*}
Moreover, by (\ref{estimate of derivative}), for $k\in [1,n]$, $n\in[1,N-1]$, $i\in \mathbb{N}$,
\begin{eqnarray*}\max_{t\in B(0,2)}\,\|d^kg_{n,i}(t)\|&\leq&
b_n^{-k+1}\max_{y\in B(f^{n(i-1)}(x),2b_n^{-1})}\|d^kf^n(y)\|\\[2mm]
&\leq&
\max_{y\in B(f^{n(i-1)}(x),2b_n^{-1})}\|d^kf^n(y)\|\leq N^{N+\frac12}\cdot 2^m\cdot M_0\cdot (\frac{2\sqrt{m}L_0^N}{\rho})^N.
\end{eqnarray*}
Note that $N$ is determined by (\ref{large n}), so $N$ depends only on $m$, $L_0$, $\rho$. Hence there is a constant
$C_0=C_0(m,L_0,M_0,\rho)$,  independent of $n$, such
that
\begin{eqnarray*}\max_{t\in B(0,2)}\,\|d^kg_{n,i}(t)\|&\leq&
C_0L_0^n,\quad 1\leq k\leq n,\quad i\in \mathbb{N},\quad
\forall\,n\geq 1.\end{eqnarray*}\end{proof} Given $p\in \mathbb{N}$,
define $F_i=g_{n,i}\circ \cdots\circ g_{n,1}$, $i=1,\cdots,p$, and
let
$$L_i=\{t\in B(0,1)\mid F_j(t)\in B(0,1),\,1\leq j\leq i\}.$$
 It follows that
$$g^{i}(B_{p}(g,x,\delta_n))\subset \delta_n F_i(L_i)+g^i(x),\quad \forall\, 1\leq i\leq p.$$
For any $\tau\in \mathbb{N}$, $v=(i_1,\cdots,i_m)\in \mathbb{Z}^m$,
we define an affine transformation $$w_{\tau,v}: Q_1\rightarrow
\mathbb{R}^m, \quad z\rightarrow (z+v)/\tau.$$
Recall that we considered $B(x,\delta_n)\subset \mathbb{R}^m$ by taking analytic charts. The ball
$B(x,\delta_n)$ is covered by at most
$\xi_{\delta_n}:=([2\delta_n\tau]+2)^m$ \, subcubes $w_{\tau,v}(
Q_1)$. For each such subcube, we define
$\sigma_v(z)=\delta_n^{-1}(w_{\tau,v}(z)-x)$. Then $d^k
(\sigma_v)=0,\quad 2\leq k<\infty$. Choosing $\tau>1$ large, we can
suppose that $\sigma_v(Q_1)\subset B(0,2)$ and
\begin{eqnarray}\label{tau large}\|\sigma_v\|_1\leq
\delta_n^{-1}\tau^{-1}<1.\end{eqnarray}

Write $C_n=C_0L_0^n$.

\begin{Prop}\label{Yomdin} Under the assumptions of (\ref{large n}), take $\tau$
large satisfying (\ref{tau large}). For each $\sigma_v$ and $s\in
[1,n]$, there exists a family of $C^s$ maps $\{\psi_{p,j}:
Q_{1}\rightarrow Q_1,\quad 1\leq j\leq \kappa^p\}$, $p\in \mathbb{N}$,
with $\kappa=\kappa(s,m,C_n)=\mu(s,m)C_n^{\frac{m}{s}}$,
 satisfying the following properties:

\begin{itemize}
\item
$F_p(L_p)\subset \cup_{1\leq j\leq \kappa^p}\,F_p\circ\sigma_v\circ
\psi_{p,j}(Q_1)\subset
B(0,2)$;\\
\item $\|\psi_{p,j}\|_s\leq 1$;\\
\item $ \|F_p\circ \sigma_v\circ  \psi_{p,j}\|_s\leq 1$,
$j=1,\cdots,\kappa^p$;\\
\item For any $p\geq2$, $i\in \{1,\cdots, \kappa^p\} $ there exist $j\in \{1,\cdots,\kappa^{p-1}\}$ and a map
$\phi_{p,i}^{p-1,j}$ with $\|\phi_{p,i}^{p-1,j}\|_s\leq 1$ such that
$$\psi_{p,i}=\psi_{p-1, j}\circ \phi_{p,i}^{p-1,j}.$$

\end{itemize}
The constant $\mu=\mu(s,m)$ depends only on $s$ and $m$, but it is independent
of $n$.
\end{Prop}

The above proposition is an application of the main corollary in Chapter
3.5  of Gromov\cite{Gr85}, which improved the estimates in Lemma
2.3 of Yomdin \cite{Yom87}. The results of \cite{Gr85} in fact hold
for any $C^s$ map but here we  only use for $s\in[1, n]$. The
additional contracting properties in the norm $\|\cdot\|_s$   were
observed by Buzzi \cite{Buz97}, without increase of $\kappa$ since
his analysis based on the previous constants in \cite{Gr85}.

Now we begin to prove Theorem \ref{Thm1}. For $n\in \mathbb{N}$, we
can choose $s(n)\in [1,n]$ such that\begin{eqnarray}
\label{increasing} && s(1)\leq s(2)\leq \cdots\leq
s(n)\leq\cdots,\quad s(n+1)-s(n)\leq 1,\quad \lim_{n\to
\infty}s(n)=+\infty,\\[2mm]
\label{slow }&&  \lim_{n\to \infty}\frac{\mu(s(n),m+1)}{n^{1/2}}=0.
\end{eqnarray}
So, there exists $\lambda_0>0$ such that
$\frac{\mu(s(n),m+1)}{n^{1/2}}\leq \lambda_0$ for all
$n\in\mathbb{N}$.

Redefining $\mu(s(n),m)$ as $\mu'(s(n),m)=\max_{1\leq k \leq m} \mu(s(n),k)$,
 we can suppose that $\mu(s(n),m)$ is non-decreasing with respect to $m$. By (\ref{slow }), it follows that
\begin{eqnarray} \label{converging} \lim_{n\to
\infty}\frac{\mu(s(n),m)}{n^{1/2}}=0.
\end{eqnarray}

Given $\vep_1>0$, take $\vep$ small (depending on $n$)  so that for any $x_1,x_2\in M$ with $d(x_1,x_2)<\vep$, one has
\begin{eqnarray}\label{small distance}d(f^{i}x_1,f^{i}x_2)<\vep_1,\quad
\forall\, i=0,\cdots,n-1.\end{eqnarray} Recall that we can use at most $\xi_{\delta_n}$ subcubes $w_{\tau,v_1}(Q_1),\cdots,w_{\tau,v_{\zeta}}(Q_1) (\zeta\leq \xi_{\delta_n})$ to cover $B(x,\delta_n)$. Now fix
$T$ to be one $\vep$-dense subset of $Q_1$. Let
$$R=x+\delta_n(\cup_{1\leq \eta\leq \zeta}\cup_{1\leq i\leq \kappa^{p-1}}\,\sigma_{v_{\eta}}\circ\psi_{p-1,i}(T)).$$
  We claim that $R$ is a
$(p,\vep)$-spanning set of $g$ restricted to $B_p(g,x,\delta_n)$.
Indeed, for any $y\in B_p(g,x,\delta_n)$ there exist $\eta\in [1,\zeta]$,  $t\in Q_1$ and
$i\in \{1,\cdots,\kappa^{p-1}\}$
 such that
 $$y=\delta_n\sigma_{v_{\eta}}\circ\psi_{p-1,i}(t)+x.$$
We can choose $c\in T$ satisfying $|t-c|\leq \vep$. Denote $c'=
\delta_n\sigma_{v_{\eta}}\circ\psi_{p-1,i}(c)+x$. Then  $d(y,c')\leq |t-c|\leq \vep$ and, for every
$q=1,\cdots,p-1$ we have
$$g^q(y)=g^q(\delta_n\sigma_{v_{\eta}}\circ\psi_{p-1,i}(c)+x)=\delta_n F_q\circ \sigma_{v_{\eta}}\circ\psi_{p-1,i}(c)+g^q(x).$$
Notice that $\psi_{p-1,i}=\psi_{q,j}\circ \phi_{p-1,i}^{q,j}$ for
some $j$. Moreover, the maps $\phi_{p-1,i}^{q,j}$ and $\delta_n
F_q\circ \sigma_{v_{\eta}}\circ \psi_{q,j}$   have norms $\leq 1$ in $\|
\cdot\|_s$.  Thus,
$$|g^q(y)-g^q(b')|\leq |t-c|\leq \vep.$$
So $R$ is a $(p,\vep)$-spanning set of $g$ restricted to
$B_p(g,x,\delta_n)$. Observing that $\zeta\leq \xi_{\delta_n}$,  we can deduce that
$$r_p(g,B_{p}(g,x,\delta_n),\vep)\leq \# R\leq
\xi_{\delta_n}\kappa^{p-1}\#
T=\xi_{\delta_n}(\mu(s(n),m)(C_n)^{\frac{m}{s(n)}})^{p-1}\# T.$$
Since $B_{pn}(f,x,\delta_n)\subset
B_{p}(f^n,x,\delta_n)=B_{p}(g,x,\delta_n)$, we have that
$$ r_{p}(f^n, B_{pn}(f,x,\delta_n),\vep)\leq r_p(f^n,
B_{p}(f^n,x,\delta_n),\vep/2).
$$
Moreover, (\ref{small distance}) yields that $r_{pn}(f,
B_{pn}(f,x,\delta_n),\vep_1)\leq r_{p}(f^n,
B_{pn}(f,x,\delta_n),\vep)$. Note that $\xi_{\delta_n}$ and $\#
T$ are independent of $p$. Thus,
\begin{eqnarray*}h(f,B_{\infty}(f,x,\delta_n),\vep_1)&=&\limsup_{p\rightarrow+\infty}\frac {1}{pn}\log
r_{pn}(f,B_{\infty}(f,x,\delta_n),\vep_1)\\[2mm]
&\leq& \limsup_{p\rightarrow
+\infty}\frac {1}{pn}\log r_p(f^n,B_{p}(f^n,x,\delta_n),\vep/2)\\[2mm]&\leq&
\frac1n\Big{(}\frac{m}{s(n)}\log
C_n+\log\mu(s(n),m)\Big{)},\end{eqnarray*} which, together with the
arbitrariness of $\vep_1$ and $x$, implies that
\begin{eqnarray}\label{local entropy}h^{*}(f,\delta_n)\leq
 \frac1n\Big{(}\frac{m}{s(n)}\log C_n+\log\mu(s(n),m)\Big{)}.\end{eqnarray}In addition, $C_n=C_0L_0^n$,
so\begin{eqnarray}h^{*}(f,\delta_n)&\leq&
\frac{m}{s(n)}(\frac{1}{n}\log C_0+\log
L_0)+\frac{1}{n}\log\mu(s(n),m)\nonumber\\[2mm]\label{local estimate}&\leq&
\frac{m}{s(n)}(\log C_0+\log
L_0)+\frac{1}{n^{1/2}}\cdot\frac{1}{n^{1/2}}\log\mu(s(n),m).\end{eqnarray}
Noting that  $\delta_n=(L_0^nn^2)^{-1}$ is decreasing in $n$,
we  can define $n(\delta):=\max\{k: \delta_k \geq \delta\}$ for $\delta\in (0,\delta_1]$. Hence, $s=s(n(\delta))$ can be
considered as a function  with variable $\delta$. Since
 $L_0>1$,
$$\frac1n\leq \frac{\log (e^2L_0)}{-\log \delta_n}.$$ Define $a:
(0,+\infty)\to [0,+\infty)$ as follows
\begin{equation*}a(t)=\begin{cases}1,\quad &t>\delta_1;\\[2mm]
\frac{1}{s(n)}+\frac{1}{(-\log{\delta_n})^{1/2}},\quad&
\delta_{n+1}< t\leq \delta_n.\end{cases}
\end{equation*}

\begin{Lem}\label{small expanding}For any analytic map
$f$ of $M$ with $\dim M=m$ or $m+1$, and $L(f)< L_0$,  we have
\begin{eqnarray*} h^{*}(f,\delta)&\leq&
 C(f) a(\delta),\end{eqnarray*}
where $C(f)$ is a constant depending on $f$.\end{Lem}\begin{proof}In
this case, we can take $N_1\in
 \mathbb{N}$ large enough such that for any $n\geq N_1$, (\ref{large n}) holds for $\rho=\rho(f)$, $M_0=M_0(f)$ and moreover,
$$\delta_n=n^{-2} L_0^{-n}=n^{-2} L(f)^{-n\frac{\log L_0}{\log L(f)}}\leq \rho(f)L(f)^{-n}.$$
We can therefore compute the  tail entropy $h^{*}(f,\delta_n)$ in analytic charts when $n\geq N_1$. Using the estimates (\ref{local estimate}) and (\ref{slow }),
\begin{eqnarray*}
h^{*}(f,\delta_n)&\leq&
\frac{m+1}{s(n)}(\log C_0(f)+\log
L_0)+\frac{1}{n^{1/2}}\cdot\frac{1}{n^{1/2}}\log\mu(s(n),m+1)\nonumber\\[2mm]
&\leq&\frac{m+1}{s(n)}(\log C_0(f)+\log
L_0)+\Big{(}\frac{\log (e^2L_0)}{-\log \delta_n}\Big{)}^{\frac12}\cdot \lambda_0\nonumber\\[2mm]
 &\leq & C_1(f) a(\delta_n),\end{eqnarray*}
  where  $C_0(f)=C_0(m+1,L_0,M_0(f),\rho(f))$ as in Lemma 2.1
and  $C_1(f)=\max\{(m+1)(\log C_0(f)+\log L_0), (\log(e^2L_0))^{\frac12}\cdot \lambda_0\}$. One can see that $C_0(f)$ and $C_1(f)$
  depend only on the complexificaction of $f$. Furthermore,  we take
  $$C_2(f)=\max\{\log L(f),\log L(f)/a(\delta_1), \cdots, \log L(f)/a(\delta_{N_1-1})\}.$$ Writing $C(f)=\max\{C_1(f),C_2(f)\}$,
   we have that for any $\delta\in (\delta_{n+1},\delta_n]$,
   $n\in \mathbb{N}$, $$h^{*}(f,\delta)\leq h^{*}(f,\delta_n)\leq C(f)a(\delta_n)=C(f)a(\delta).$$
   Moreover, when $\delta>\delta_1$,  $h^*(f,\delta)\leq h(f)\leq \log L(f)\leq C(f)$.
\end{proof}

Next we prove Theorem \ref{Thm1} for any real analytic map $f$ on $M$
without restriction on $L(f)$. Consider the product $M_1 =
\mathbb{R}\times M$. There is a simple flow $\phi(t, (s, x)) = (t +
s, x)$ on $ M_1$ called the horizontal flow. We can obtain a
suspension manifold $\widetilde{M}$ by identifying the points $(t +
1, x)$ with $(t, f(x))$. That is, we define an equivalence relation
$\sim$ on $M_1$ by $(t, x)\sim (t_0, x_0) $ iff $t_0 = t + n$ and
$x_0 = f^n(x)$. The quotient space $\widetilde{M} = M_1/\sim$ is
also a real analytic manifold and $\dim \widetilde{M} = \dim M + 1$. The
horizontal flow $\phi$ pushes down to a real  analytic flow $\psi$ on
$\widetilde{M}$. To reduce the proof  to the case that the Lipschitz constant of the complexification is smaller than
  $L_0$, we can choose $i\in \mathbb{N}$ large so
that $\psi^{1/i}$ satisfies
$$L(\psi^{1/i})<L_0.$$
Then by Lemma \ref{small expanding} we have
\begin{eqnarray*}
h^{*}(\psi^{1/i},\delta)&\leq&
 C_2 a(\delta),\quad \forall\,\delta>0,\end{eqnarray*}
where $C_2$ is a constant depending on $\psi^{1/i}$. The time one
map $\psi^1$ of $\psi$ on the section $\{0\}\times M$ is smooth
conjugate to $f$. Therefore,
\begin{eqnarray}\label{estimate 1}
h^{*}(f,\delta)\leq h^{*}(\psi^1,\delta)\leq
ih^{*}(\psi^{1/i},\delta L_0^i)&\leq&
 iC_2 a(\delta L_0^i).\end{eqnarray}
Notice that $\delta_n=n^{-2}L_0^{-n}$, so for $n>i$,
$$\delta_nL_0^i=n^{-2}L_0^{-n+i}\leq \delta_{n-i}.$$
Moreover, $0\leq s(n)-s(n-i)\leq i$ and $\lim_{n\to
\infty}s(n)=\infty$.   Hence  there exists $C_3>0$ such that for $n>i$,
\begin{eqnarray*}\frac{a(\delta_n L_0^i)}{a(\delta_n)}&\leq&
\frac{a(\delta_{n-i})}{a(\delta_n)}= \frac{\frac{1}{s(n-i)}
+\frac{1}{(-\log{\delta_{n-i}})^{1/2}}}{\frac{1}{s(n)}+\frac{1}{(-\log{\delta_n})^{1/2}}}\\[2mm]&\leq&\frac{\frac{1}{s(n)-i}
+\frac{1}{(\log((n-i)^{-2}L_0^{n-i}))^{1/2}}}{\frac{1}{s(n)}+\frac{1}{(\log(n^{-2}L_0^n))^{1/2}}}\\[2mm]
&\leq& C_3. \end{eqnarray*}
Therefore, from (\ref{estimate 1}) we deduce that
\begin{eqnarray*} h^{*}(f,\delta_n)\leq
 iC_2C_3 a(\delta_n),\quad \forall\,n>i.\end{eqnarray*}
Taking $C(f)=\max\{iC_2C_3,\log L(f),\log L(f)/a(1),\cdots, \log L(f)/a(i)\}$, we complete the proof of Theorem  \ref{Thm1}.  Note that from this proof we see that $C(f)$
depends on $i$, $C_2$, $C_3$ and $L(f)$, so it can be chosen locally
uniformly with respect to the complexification topology. \hfill $\Box$

\bigskip


\end{document}